\documentclass[9pt,a4paper,draft]{article}
\usepackage{bbm}
\usepackage{pifont}
\usepackage{bbding}
\usepackage{amsmath}
\usepackage{mathrsfs}
\usepackage{amsfonts}
\usepackage{amssymb}
\usepackage{amsfonts,amssymb,amsmath,indentfirst,amsthm}
\usepackage{epsfig}
\usepackage[numbers,sort&compress]{natbib}
\usepackage{mathtools}
\def\ess~inf{\mathop{\rm ess~inf}}

\numberwithin{equation}{section}

\newenvironment{key words}{\emph{\texttt{Keywords}}\mbox{  }}{ }
\newtheorem{theorem}{Theorem}[section]
\newtheorem{lemma}[theorem]{Lemma}

\renewenvironment{proof}{\noindent{\textbf{Proof.}}}{\hfill$\Box$}
\theoremstyle{remark}

\theoremstyle{plain}

\makeatletter

\newcommand{\Rmnum}[1]{\expandafter\@slowromancap\romannumeral #1@}
\makeatother

\begin{document}
\setlength{\baselineskip}{1.2\baselineskip}
\title
{\textbf{Convergence problems along curves for generalized Schr\"{o}dinger operators with polynomial growth} \thanks{This work is supported by Natural Science Foundation of China (No.11601427);
China Postdoctoral Science Foundation (No.2017M613193);  Natural Science Basic Research Plan in Shaanxi Province of China (No.2017JQ1009).}}

\author{{Wenjuan Li \qquad \qquad Huiju Wang }  \\
(\small{School of Mathematics and Statistics, Northwestern Polytechnical University,
 Xi'an, 710129, China})\\    }
\date{}
 \maketitle
{\bf Abstract:}\  In this paper we  build the relationship between  smoothness of the functions and convergence rate along curves for a class of generalized Schr\"{o}dinger operators with polynomial growth. We show that the convergence rate  depends only on the growth condition of the phase function and regularity of the curve. Our result can be applied to a wide class of operators. In particular,  convergence results along curves for a class of generalized Schr\"{o}dinger operators with non-homogeneous phase functions is built and then the convergence rate is established.



\section{\textbf{Introduction}\label{Section 1}}
Consider the generalized Schr\"{o}dinger equation
\begin{equation}\label{Eq1.1}
\begin{cases}
 \partial_{t}u(x,t)-iP(D)u(x,t) =0 \:\:\:\ x \in \mathbb{R}^{n}, t \in \mathbb{R}^{+},\\
u(x,0)=f \\
\end{cases}
\end{equation}
where $D=\frac{1}{i}(\frac{\partial}{\partial x_{1}},\frac{\partial}{\partial x_{2}},...,\frac{\partial}{\partial x_{n}})$, $P(\xi)$ is a real continuous function defined on $\mathbb{R}^{n}$, $P(D)$ is defined via its real symbol
\[P(D)f(x)= \int_{\mathbb{R}^{n}}{e^{ix \cdot \xi}P(\xi)\hat{f}(\xi)d\xi}.\]
 The solution of (\ref{Eq1.1}) can be formally written as
\begin{equation}
e^{itP(D)}f(x):= \int_{\mathbb{R}^{n}}{e^{ix \cdot \xi +itP(\xi)} \hat{f} (\xi)d\xi },
\end{equation}
where $\hat{f}$ denotes the Fourier transform of $f$.

The Carleson's problem, that is. to determine the optimal $s$ for which
\begin{equation}
\mathop{lim}_{t \rightarrow 0^{+}} e^{itP(D)}f(x) = f(x)
\end{equation}
almost everywhere whenever $f \in H^{s}(\mathbb{R}^{n})$, has been widely studied since the first work by Carleson (\cite{C}), later works see \cite{L}, \cite{V}, \cite{SS}, \cite{S}, \cite{Miao} and references therein. Sharp results were derived in some cases, such as  the elliptic case (\cite{DGL,DZ}, when $n \ge 1$, $P(\xi)=|\xi|^{2}$); the non-elliptic case (\cite{KPV}, when $n \ge 1$, $P(\xi)=\xi_{1}^{2}-\xi_{2}^{2} \pm \cdot \cdot \cdot \pm \xi_{n}^{2}$) and the fractional case (\cite{CK}, when $n \ge 1$ and $P(\xi)= |\xi|^{\alpha}$, $\alpha >1$).

A natural generalization of the convergence problem is to consider almost everywhere convergence along variable curves instead of vertical lines. Let $\gamma$ be a continuous function such that
\[\gamma(x,t): \mathbb{R}^{n} \times [0,1] \rightarrow \mathbb{R}^{n}, \gamma(x,0)=x.\]
Consider the pointwise convergence problem along the curve $(\gamma(x,t),t)$, that is, to determine the optimal $s$ for which
\begin{equation}
\mathop{lim}_{t \rightarrow 0^{+}} e^{itP(D)}(f)(\gamma (x,t)) = f(x)
\end{equation}
almost everywhere whenever $f \in H^{s}(\mathbb{R}^{n})$. This problem has been considered by \cite{CLV, LR, DN1, DN2,LW1}. When $\gamma(x,t)$ is H\"{o}lder continuous with respect to $t$ uniformly for each $x$ and bilipschitz with respect to $x$ for each $t\in[0,1]$, sharp results were obtained by \cite{CLV} for $n=1$, $P(\xi)= |\xi|^{2}$. For smooth curve $\gamma(x,t)\in C^1(\mathbb{R}^n\times\mathbb{R})$, since the corresponding convergence result is equivalent to the pointwise convergence of solutions to the free Schr\"{o}dinger equation with initial data, the convergence rate follows from \cite{LW3}.

However, how is the relationship between smoothness of the functions and the convergence rate along the curves with less smooth condition?
Hence, it is interesting to seek the convergence rate of $e^{itP(D)}f(x)$ as $t$ tends to $0$ along the curve $(\gamma(x,t),t)$ if $f$ has more regularity. The problem is, suppose that $e^{itP(D)}(f)(\gamma(x,t))$ converge pointwisely to $f(x)$ for $f \in H^{s}(\mathbb{R}^{n})$ as $t$ tends to $0$, whether or not it is possible that, for $f \in H^{s + \delta}(\mathbb{R}^{n})$, $\delta \ge 0$,
\begin{equation}
e^{itP(D)}(f)(\gamma(x,t)) - f(x) = o(t^{\theta(\delta)})
\end{equation}
almost everywhere for some $\theta(\delta) \ge 0$? Cao, Fan and Wang \cite{CFW} proved this property in the vertical case $\gamma(x,t)=x$ when $n \ge 1$, $P(\xi) = |\xi|^{2}$, and when $n = 1$, $P(\xi)=|\xi|^{\alpha}, \alpha >1$. Authors of this paper improved the results in \cite{CFW} to general $P(\xi)$ with polynomial growth. It is proved in \cite{LW3} that the convergence rate in the vertical case depends only on the growth condition of $P(\xi)$.

In this paper, we first discuss the convergence rate for a class of Schr\"{o}dinger operators with polynomial growth along curves. Denote by $B(x_{0},R)$ the ball with center $x_{0} \in \mathbb{R}^{n}$ and radius $R \lesssim 1$. Suppose that $\gamma(x,t)$ satisfies
\begin{equation}
|\gamma(x,t)-\gamma(x,t^{\prime})| \le C|t-t^{\prime}|^{\alpha}, \:\ 0<\alpha \le 1
\end{equation}
uniformly for $x \in B(x_{0},R)$ and $t,t^{\prime} \in [0,1]$. Then our first main result is as follows:
\begin{theorem}\label{theorem1.1}
If there exist $m \ge 1$, $s_{0} >0$ such that
\begin{equation}\label{Eq1.7}
|P(\xi)| \lesssim |\xi|^{m}, |\xi| \rightarrow +\infty,
\end{equation}
and for each $s > s_{0}$,
\begin{equation}\label{Eq1.8}
\biggl\|\mathop{sup}_{0<t<1} |e^{itP(D)}(f)(\gamma(x,t))|\biggl\|_{L^{p}(B(x_{0},R))} \lesssim \|f\|_{H^{s}(\mathbb{R}^{n})},  \:\ p \ge 1,
\end{equation}
then for all $f \in H^{s+\delta}(\mathbb{R}^{n})$, $0 \le \delta <m$,
\begin{equation}\label{Eq1.9}
e^{itP(D)}(f)(\gamma(x,t)) - f(x) = o(t^{\alpha \delta /m}), \:\ a.e. \:\:\ x \in B(x_{0},R) \text{\quad as \quad} t \rightarrow 0^{+}.
\end{equation}
\end{theorem}

Note that the convergence rate in Theorem \ref{theorem1.1} depends on the growth condition of the phase function and the regularity of the curve, but independent of the gradient of the phase function and the dimension of the spatial space. In particular, for non-zero Schwartz functions, the convergence rate seems no faster than $t^{\alpha}$ as $t$ tends to $0$ along the curve $(\gamma(x,t),t)$, see Theorem \ref{theorem2.1} below. Theorem \ref{theorem1.1} is quite general and can be applied to a wide class of operators, such as the non-elliptic Schr\"{o}dinger operators ($P(\xi)=\xi_{1}^{2}-\xi_{2}^{2} \pm \cdot \cdot \cdot \pm \xi_{n}^{2}$), the fractional Schr\"{o}dinger operators ($P(\xi)= |\xi|^{\alpha}$, $\alpha >1$). Therefore, the convergence rate along curves can be deduced from Theorem \ref{theorem1.1} provided that the pointwise convergence along curves is proved.

Next, we consider a class of operators with non-homogeneous phase function and obtained the corresponding convergence result and convergence rate. For convenience, we concentrate ourselves on the case $n=2$ and consider a class of operators with phase function
\[P_{m_{1},m_{2}}(\xi) = \xi_{1}^{m_{1}} \pm \xi_{2}^{m_{2}},\]
where $m_{1}, m_{2} \in \mathbb{N}^{+}$, $2 \le m_{1} \le m_{2}$, and a class of curves $\gamma(x,t): \mathbb{R}^{2} \times [0,1]\rightarrow \mathbb{R}^{2}$, $\gamma(x,0)=x$ satisfies
\begin{equation}\label{Eq1.101}
|\gamma(x,t)-\gamma(y,t)| \sim |x-y|,
\end{equation}
\begin{equation}\label{Eq1.111}
|\gamma(x,t)-\gamma(x,t^{\prime})| \lesssim |t-t^{\prime}|^{\frac{1}{m_{1}-1}}
\end{equation}
for each $x,y \in B(x_{0},R)$ and $t,t^{\prime} \in [0,1].$ We have the following result:

\begin{theorem}\label{theorem1.3}
For each $s> 1/2$,
\begin{equation}\label{Eq1.10}
\biggl\|\mathop{sup}_{0<t<1} |e^{itP_{m_{1}, m_{2}}(D)}(f)(\gamma(x,t))|\biggl\|_{L^{2}(B(x_{0},R))} \lesssim \|f\|_{H^{s}(\mathbb{R}^{2})}
\end{equation}
whenever $f \in H^{s}(\mathbb{R}^{2})$.  Moreover, by Theorem \ref{theorem1.1}, for all $f \in H^{s+\delta}(\mathbb{R}^{2})$, $0 \le \delta <m_{2}$,
\begin{equation}\label{Eq1.11}
e^{itP_{m_{1},m_{2}}(D)}(f)(\gamma(x,t)) - f(x) = o(t^{\delta /(m_{1}-1)m_{2}}), \:\ a.e. \:\:\ x \in B(x_{0},R) \text{\quad as \quad} t \rightarrow 0^{+}.
\end{equation}
\end{theorem}

In the rest of the introduction, we briefly sketch the proof of Theorem \ref{theorem1.3}. We only need to prove (\ref{Eq1.10}), (\ref{Eq1.11}) then follows from Theorem \ref{theorem1.1}. Recall that Theorem 4.1 in \cite{KPV} by Kenig-Ponce-Vega shows that for each $s> 1/2$,
\begin{equation}\label{Eq1.12}
\biggl\|\mathop{sup}_{0<t<1} |e^{itP_{m_{1}, m_{2}}(D)}f|\biggl\|_{L^{2}(B(0,1))} \lesssim \|f\|_{H^{s}(\mathbb{R}^{2})}
\end{equation}
whenever $f \in H^{s}(\mathbb{R}^{n})$. (\ref{Eq1.10}) follows from (\ref{Eq1.12}) and Theorem \ref{theorem1.2} below.
\begin{theorem}\label{theorem1.2}
If there exists $s_{0} > 0$ such that for each $s> s_{0}$,
\begin{equation}\label{Eq1.13}
\biggl\|\mathop{sup}_{0<t<1} |e^{itP_{m_{1}, m_{2}}(D)}f|\biggl\|_{L^{p}(B(0,1))} \lesssim \|f\|_{H^{s}(\mathbb{R}^{2})}
\end{equation}
whenever $f \in H^{s}(\mathbb{R}^{2})$, then for each $s> s_{0}$,
\begin{equation}\label{Eq1.14}
\biggl\|\mathop{sup}_{0<t<1} |e^{itP_{m_{1}, m_{2}}(D)}(f)(\gamma(x,t))|\biggl\|_{L^{p}(B(x_{0},R))} \lesssim \|f\|_{H^{s}(\mathbb{R}^{2})}
\end{equation}
whenever $f \in H^{s}(\mathbb{R}^{2})$.
\end{theorem}

When $m_{1}=m_{2}$, result as in Theorem \ref{theorem1.2} was first obtained by  Cho, Lee and Vargas (\cite{CLV}, Proposition 4.3). In order to prove Theorem \ref{theorem1.2}, after Littlewood-Paley decomposition, we only need to  consider $f$, suppp$\hat{f} \subset \{\xi: |\xi| \sim \lambda\}$, $\lambda \gg 1$. We decompose [0,1] into bounded overlap intervals $\mathfrak{J}= \cup_{J \in \mathfrak{J}}J$, each $J$ is of length $\lambda^{1-m_{1}}$. For each $J$, it follows from (\ref{Eq1.13}) that
\begin{equation}\label{Eq1.15}
\biggl\|\mathop{sup}_{t \in J} |e^{itP_{m_{1}, m_{2}}(D)}f|\biggl\|_{L^{p}(B(x_{0},R))} \lesssim \lambda^{s_{0}+\epsilon} \|f\|_{L^{2}}, \:\ \forall \epsilon >0
\end{equation}
As Lemma 2.2 in \cite{LR}, inequalities (\ref{Eq1.15}),  (\ref{Eq1.101}), (\ref{Eq1.111}) imply
\begin{equation}\label{Eq1.16}
\biggl\|\mathop{sup}_{t \in J} |e^{itP_{m_{1}, m_{2}}(D)}(f)(\gamma(x,t))|\biggl\|_{L^{p}(B(x_{0},R))} \lesssim \lambda^{s_{0}+\epsilon} \|f\|_{L^{2}}, \:\ \forall \epsilon >0.
\end{equation}
Inequality (\ref{Eq1.14}) then follows from (\ref{Eq1.16}) and a time localizing lemma:

\begin{theorem}\label{theorem3.1}
Let $\mathfrak{J}=\{J\}$ be a collection of intervals of length $\lambda^{1-m_{1}}$ with bounded overlap, $[0,1]=\bigcup_{J \in \mathfrak{J}}J$. Suppose that for some $\alpha >0$, $p \ge 2$,
\begin{equation}\label{Eq2.3.1}
\biggl\|\mathop{sup}_{t \in J} |e^{itP_{m_{1}, m_{2}}(D)}(f)(\gamma(x,t))|\biggl\|_{L^{p}(B(x_{0},R))} \lesssim \lambda^{\alpha} \|f\|_{L^{2}}
\end{equation}
provided that suppp$\hat{f} \subset \{\xi: |\xi| \sim \lambda\}$, then for any $\epsilon >0$, we have
\begin{equation}\label{Eq2.3.1}
\biggl\|\mathop{sup}_{0<t<1} |e^{itP_{m_{1}, m_{2}}(D)}(f)(\gamma(x,t))|\biggl\|_{L^{p}(B(x_{0},R))} \lesssim \lambda^{\alpha+\epsilon} \|f\|_{L^{2}}
\end{equation}
whenever suppp$\hat{f} \subset \{\xi: |\xi| \sim \lambda\}$.
\end{theorem}
Based on our previous argument, we will omit some details and only prove Theorem \ref{theorem3.1} in Section 3.

\section{Proof of Theorem \ref{theorem1.1}}\label{Section 2}
We first prove the following Lemma \ref{lemma2.1}.
\begin{lemma}\label{lemma2.1}
Assume that $g$ is a Schwartz function whose Fourier transform is supported in the annulus $A(\lambda)=\{\xi \in \mathbb{R}^{n}: |\xi| \sim \lambda\}$. $\gamma(x,t)$ satisfies
\[|\gamma(x,t)- x| \lesssim t^{\alpha},\hspace{0.8cm} \gamma(x,0)=0\]
for all $x \in B(x_{0},R)$ and $t \in (0, \lambda^{-\frac{1}{\alpha}})$. Then for each $x \in B(x_{0},R)$ and $t \in (0, \lambda^{-\frac{1}{\alpha}})$,
\begin{align}\label{Eq2.1}
|e^{itP(D)}g(\gamma(x,t))| \le \sum_{\mathfrak{l} \in \mathbb{Z}^{n}}{\frac{C_{n}}{(1+|\mathfrak{l}|)^{n+1}}\biggl|\int_{\mathbb{R}^{n}}{e^{i(x+ \frac{\mathfrak{l}}{\lambda})\cdot\xi+itP(\xi)}\hat{g}}(\xi)d\xi \biggl|}.
\end{align}
\end{lemma}

\begin{proof}
As in \cite{LR}, we introduce a cut-off function $\phi$ which is smooth and equal to $1$ on $B(0,2)$ and supported on $(-\pi, \pi)^{n}$. After scaling we have
\begin{align}
e^{itP(D)}(g)(\gamma(x,t))&= \lambda^{n} \int_{\mathbb{R}^{n}}{e^{i\lambda \gamma(x,t)\cdot \eta + itP(\lambda \eta)} \phi(\eta)\hat{g}(\lambda \eta) d\eta} \nonumber\\
&= \lambda^{n} \int_{\mathbb{R}^{n}}{e^{i\lambda \gamma(x,t)\cdot \eta -i\lambda x \cdot \eta  +i\lambda x \cdot \eta + itP(\lambda \eta)} \phi(\eta)\hat{g}(\lambda \eta) d\eta}.
\end{align}
Since
\[|\lambda\gamma(x,t)-\lambda x| \lesssim 1,\]
then by Fourier expansion,
\[\phi(\eta)e^{i\lambda[\gamma(x,t)-x] \cdot \eta} = \sum_{\mathfrak{l} \in \mathbb{Z}^{n}}{c_{\mathfrak{l}}(x,t)e^{i\mathfrak{l}\cdot \eta}},\]
where
\[|c_{\mathfrak{l}}(x,t)| \lesssim \frac{C_{n}}{(1+|\mathfrak{l}|)^{n+1}}\]
uniformly for each $\mathfrak{l} \in \mathbb{Z}^{n}$, $x \in B(x_{0},R)$ and $t \in (0, \lambda^{-\frac{1}{\alpha}})$. Then we have
\begin{align}
|e^{itP(D)}g(\gamma(x,t))| &\le \sum_{\mathfrak{l} \in \mathbb{Z}^{n}}{\frac{C_{n}\lambda^{n}}{(1+|\mathfrak{l}|)^{n+1}}\biggl|\int_{\mathbb{R}^{n}}{e^{i \mathfrak{l} \cdot \eta  +i\lambda x \cdot \eta + itP(\lambda \eta)} \hat{g}(\lambda \eta) d\eta} \biggl|} \nonumber\\
&= \sum_{\mathfrak{l} \in \mathbb{Z}^{n}}{\frac{C_{n}}{(1+|\mathfrak{l}|)^{n+1}}\biggl|\int_{\mathbb{R}^{n}}{e^{i \frac{\mathfrak{l}}{\lambda} \cdot \xi  +i x \cdot \xi + itP(\xi)} \hat{g}(\xi) d\xi} \biggl|}, \nonumber
\end{align}
then we arrive at (\ref{Eq2.1}).
\end{proof}

\textbf{Proof of Theorem \ref{theorem1.1}.}
It suffices to show that for some $q \ge 1$ and $\forall \epsilon >0$,  $s_{1} = s_{0} + \epsilon$,
\begin{equation}\label{Eq2.3}
\biggl\|\mathop{sup}_{0<t<1} \frac{|e^{itP(D)}(f)(\gamma(x,t))-f(x)|}{t^{\alpha\delta/m}}\biggl\|_{L^{q}(B(x_{0},R))} \lesssim \|f\|_{H^{s_{1}+ \delta}(\mathbb{R}^{n})}.
\end{equation}

In fact, if (\ref{Eq2.3}) holds, then fix $\lambda >0$,  choose $g \in C_{c}^{\infty}(\mathbb{R}^{n})$ such that
\[\|f-g\|_{H^{s_{1}+ \delta}(\mathbb{R}^{n})} \le  \frac{\lambda \epsilon^{1/q}}{2},\]
it follows
\begin{align}\label{Eq2.4}
&\biggl|\biggl\{{ x \in B(x_{0},R): \mathop{sup}_{0<t<1} \frac{|e^{itP(D)}(f-g)(\gamma(x,t))-(f-g)(x)|}{t^{\alpha\delta/m}}} > \frac{\lambda}{2}\biggl\}\biggl|  \nonumber\\
&\le \frac{2^{q}}{\lambda^{q}} \biggl\|\mathop{sup}_{0<t<1} \frac{|e^{itP(D)}(f-g)(\gamma(x,t))-(f-g)(x)|}{t^{\alpha\delta/m}}\biggl\|_{L^{q}(B(x_{0},R))}^{q} \nonumber\\
&\lesssim \frac{2^{q}}{\lambda^{q}}\|f-g\|_{H^{s_{1}+ \delta}(\mathbb{R}^{n})}^{q} \nonumber\\
&\le \epsilon.
\end{align}
Moreover,
\begin{align}\label{Eq2.5}
\frac{|e^{itP(D)}(g)(\gamma(x,t))-g(x)|}{t^{\alpha\delta/m}}  \rightarrow 0,\hspace{0.5cm} \textmd{if } \hspace{0.2cm}t \rightarrow 0^{+}
\end{align}
uniformly for $x \in B(x_{0},R)$. Indeed, for each $x \in B(x_{0},R)$,
\begin{align}\label{Eq2.6}
\mathop{lim}_{t \rightarrow 0^{+}}{\frac{| e^{itP(D)}(g)(\gamma(x,t))-g(x)|}{t^{\alpha\delta /m}}} &\le \mathop{lim}_{t \rightarrow 0^{+}}{\frac{| e^{itP(D)}(g)(\gamma(x,t))-g(\gamma(x,t))|}{t^{\alpha\delta /m}}} \nonumber\\
& \:\:\ + \mathop{lim}_{t \rightarrow 0^{+}}{\frac{|g(\gamma(x,t))-g(x)|}{t^{\alpha\delta /m}}}.
\end{align}
By mean value theorem, we have
\begin{align}\label{Eq2.7}
\frac{| e^{itP(D)}(g)(\gamma(x,t))-g(\gamma(x,t))|}{t^{\alpha\delta /m}} &\le t^{1-\alpha\delta/m}\int_{\mathbb{R}^{n}}{|P(\xi)||\hat{g}(\xi)|d\xi},
\end{align}
and
\begin{align}\label{Eq2.8}
\frac{|g(\gamma(x,t))-g(x)|}{t^{\alpha\delta /m}} \le \frac{|\gamma(x,t)-x|}{t^{\alpha\delta/m}}\int_{\mathbb{R}^{n}}{|\xi||\hat{g}(\xi)|d\xi} \le t^{\alpha-\alpha\delta/m}\int_{\mathbb{R}^{n}}{|\xi||\hat{g}(\xi)|d\xi}.
\end{align}
Inequalities (\ref{Eq2.6}) - (\ref{Eq2.8}) imply (\ref{Eq2.5}).

By (\ref{Eq2.4}) and (\ref{Eq2.5}) we have
\begin{align}
\biggl|\biggl\{{ x \in B(x_{0},R): \mathop{lim sup}_{t \rightarrow 0^{+}} \frac{|e^{itP(D)}(f)(\gamma(x,t))-f(x)|}{t^{\alpha \delta/m}}} > \lambda\biggl\}\biggl| \le \epsilon,
\end{align}
which  implies (\ref{Eq1.9}) for $f \in H^{s_{1} +\delta}(\mathbb{R}^{n})$ and almost every $x \in B(x_{0}, R)$. By the arbitrariness of $\epsilon$, in fact we can get (\ref{Eq1.9}) for all $f \in H^{s +\delta}(\mathbb{R}^{n})$, $s > s_{0}$. Next we will prove (\ref{Eq2.3}) for $q=min\{p,2\}$.

In order to prove (\ref{Eq2.3}), we decompose $f$ as
\[f=\sum_{k=0}^{\infty}{f_{k}},\]
where supp$ \hat{f_{0}} \subset B(0,1)$, supp$ \hat{f_{k}} \subset \{\xi: |\xi| \sim 2^{k}\}$, $k \ge 1$. It follows that
\begin{align}\label{Eq2.10}
& \biggl\|\mathop{sup}_{0<t<1} \frac{|e^{itP(D)}(f)(\gamma(x,t))-f(x)|}{t^{\alpha \delta/m}}\biggl\|_{L^{q}(B(x_{0},R))} \nonumber\\
&\le \sum_{k=0}^{\infty}{\biggl\|\mathop{sup}_{0<t<1} \frac{|e^{itP(D)}(f_{k})(\gamma(x,t))-f_{k}(x)|}{t^{\alpha \delta/m}}\biggl\|_{L^{q}(B(x_{0},R))}}.
\end{align}

For $k \lesssim 1$, because (\ref{Eq2.7}), (\ref{Eq2.8}) and $P(\xi)$ is continuous,
\begin{align}\label{Eq2.11}
&\biggl\|\mathop{sup}_{0<t<1} \frac{|e^{itP(D)}(f_{k})(\gamma(x,t))-f_{k}(x)|}{t^{\alpha \delta/m}}\biggl\|_{L^{q}(B(x_{0},R))} \nonumber\\
&\le \biggl\|\mathop{sup}_{0<t<1} \frac{|e^{itP(D)}(f_{k})(\gamma(x,t))-f_{k}(\gamma(x,t))|}{t^{\alpha \delta/m}}\biggl\|_{L^{q}(B(x_{0},R))} \nonumber\\
& \:\;\ +\biggl\|\mathop{sup}_{0<t<1} \frac{|f_{k}(\gamma(x,t))-f_{k}(x)|}{t^{\alpha \delta/m}}\biggl\|_{L^{q}(B(x_{0},R))} \nonumber\\
&\lesssim \|f\|_{H^{s_{1}+ \delta}(\mathbb{R}^{n})}.
\end{align}

For $k \gg 1$,
\begin{align}\label{Eq2.12}
&\biggl\|\mathop{sup}_{0<t<1} \frac{|e^{itP(D)}(f_{k})(\gamma(x,t))-f_{k}(x)|}{t^{\alpha \delta/m}}\biggl\|_{L^{q}(B(x_{0},R))} \nonumber\\
&\le \biggl\|\mathop{sup}_{2^{-\frac{mk}{\alpha}} \le t<1} \frac{|e^{itP(D)}(f_{k})(\gamma(x,t))-f_{k}(x)|}{t^{\alpha \delta/m}}\biggl\|_{L^{q}(B(x_{0},R))} \nonumber\\
& \:\:\ + \biggl\|\mathop{sup}_{0<t<2^{-\frac{mk}{\alpha}}} \frac{|e^{itP(D)}(f_{k})(\gamma(x,t))-f_{k}(x)|}{t^{\alpha \delta/m}}\biggl\|_{L^{q}(B(x_{0},R))} \nonumber\\
&:= I + II.
\end{align}

We first estimate $I$, from (\ref{Eq1.8}) we have
\begin{equation}\label{Eq2.13}
\biggl\|\mathop{sup}_{2^{-\frac{mk}{\alpha}}  \le t<1} |e^{itP(D)}(f_{k})(\gamma(x,t))|\biggl\|_{L^{p}(B(x_{0},R))} \lesssim 2^{(s_{0}+\frac{\epsilon}{2})k}\|f_{k}\|_{L^{2}(\mathbb{R}^{n})},
\end{equation}
hence,
\begin{align}\label{Eq2.15}
 I  &\le 2^{\delta k} \biggl\|\mathop{sup}_{2^{-\frac{mk}{\alpha}}  \le t < 1} |e^{itP(D)}(f_{k})(\gamma(x,t))-f_{k}(x)|\biggl\|_{L^{q}(B(x_{0},R))} \nonumber\\
&\le 2^{\delta k} \biggl\{\biggl\|\mathop{sup}_{2^{-\frac{mk}{\alpha}}  \le t < 1} |e^{itP(D)}(f_{k})(\gamma(x,t))|\biggl\|_{L^{q}(B(x_{0},R))} + \|f_{k}|\|_{L^{q}(B(x_{0},R))} \biggl\}\nonumber\\
&\lesssim 2^{\delta k} \biggl\{\biggl\|\mathop{sup}_{2^{-\frac{mk}{\alpha}} \le t < 1} |e^{itP(D)}(f_{k})(\gamma(x,t))|\biggl\|_{L^{p}(B(x_{0},R))} + \|f_{k}|\|_{L^{2}(B(x_{0},R))} \biggl\} \nonumber\\
&\lesssim 2^{\delta k}2^{(s_{0}+\frac{\epsilon}{2})k}\|f\|_{L^{2}(\mathbb{R}^{n})} \nonumber\\
&\lesssim 2^{-\frac{\epsilon k}{2}}\|f\|_{H^{s_{1}+ \delta}(\mathbb{R}^{n})}.
\end{align}

For $II$, by triangle inequality,
\begin{align}
II &\le \biggl\|\mathop{sup}_{0<t<2^{-\frac{mk}{\alpha}}} \frac{|e^{itP(D)}(f_{k})(\gamma(x,t))-f_{k}(\gamma(x,t))|}{t^{\alpha \delta/m}}\biggl\|_{L^{q}(B(x_{0},R))} \nonumber\\
& \:\:\ + \biggl\|\mathop{sup}_{0<t<2^{-\frac{mk}{\alpha}}} \frac{|(f_{k})(\gamma(x,t))-f_{k}(x)|}{t^{\alpha \delta/m}}\biggl\|_{L^{q}(B(x_{0},R))}.
\end{align}

Mean value theorem and Lemma \ref{lemma2.1} imply
\begin{align}\label{Eq2.16}
&\biggl\|\mathop{sup}_{0<t<2^{-\frac{mk}{\alpha}}} \frac{|f_{k}(\gamma(x,t))-f_{k}(x)|}{t^{\alpha \delta/m}}\biggl\|_{L^{q}(B(x_{0},R))} \nonumber\\
&\le  \biggl\|\mathop{sup}_{0<t<2^{-\frac{mk}{\alpha}}}  \frac{\int_{\mathbb{R}^{n}}{e^{i[\theta(x,t)x+(1-\theta(x,t))\gamma(x,t)]\cdot\xi}[\gamma(x,t)-x] \cdot \xi \hat{f_{k}}}(\xi)d\xi}{t^{\alpha\delta/m}} \biggl\|_{L^{2}(B(x_{0},R))} \nonumber\\
&\le 2^{-mk+\delta k}  \sum_{h=1}^{n}{\sum_{\mathfrak{l} \in \mathbb{Z}^{n}}{\frac{C_{n}}{(1+|\mathfrak{l}|)^{n+1}}\biggl\|\int_{\mathbb{R}^{n}}{e^{i(x+ \frac{\mathfrak{l}}{2^{k}})\cdot\xi}\xi_{h}\hat{f_{k}}}(\xi)d\xi \biggl\|_{L^{2}(B(x_{0},R))}}} \nonumber\\
&\lesssim 2^{-(m-1)k+\delta k}  \|\hat{f}_{k}\|_{L^{2}(\mathbb{R}^{n})} \nonumber\\
&\lesssim 2^{-s_{1}k}\|f\|_{H^{s_{1}+ \delta}(\mathbb{R}^{n})},
\end{align}
where $\theta(x,t) \in [0,1]$.

By Taylor's formula and Lemma \ref{lemma2.1}, we get
\begin{align}\label{Eq2.17}
&\biggl\|\mathop{sup}_{0<t<2^{-\frac{mk}{\alpha}}} \frac{|e^{itP(D)}(f_{k})(\gamma(x,t))-f_{k}(\gamma(x,t))|}{t^{\alpha \delta/m}}\biggl\|_{L^{q}(B(x_{0},R))} \nonumber\\
&\le \sum_{j=1}^{\infty}{\frac{2^{-\frac{mkj}{\alpha}+\delta k}}{j!} \biggl\|\mathop{sup}_{0<t<2^{-\frac{mk}{\alpha}}}|\int_{\mathbb{R}^{n}}{e^{i\gamma(x,t)\cdot\xi}P(\xi)^{j}\hat{f_{k}}}(\xi)d\xi|\biggl\|_{L^{q}(B(x_{0},R))}} \nonumber\\
&\le \sum_{j=1}^{\infty}{\frac{2^{-\frac{mkj}{\alpha}+\delta k}}{j!} \sum_{\mathfrak{l} \in \mathbb{Z}^{n}}{\frac{C_{n}}{(1+|\mathfrak{l}|)^{n+1}}\biggl\|\int_{\mathbb{R}^{n}}{e^{i(x+ \frac{\mathfrak{l}}{2^{mk}})\cdot\xi}P(\xi)^{j}\hat{f_{k}}}(\xi)d\xi \biggl\|_{L^{q}(B(x_{0},R))}}} \nonumber\\
&\le \sum_{j=1}^{\infty}{\frac{2^{-\frac{mkj}{\alpha}+\delta k}}{j!} \sum_{\mathfrak{l} \in \mathbb{Z}^{n}}{\frac{C_{n}}{(1+|\mathfrak{l}|)^{n+1}}\|P(\xi)^{j}\hat{f_{k}}(\xi)\|_{L^{2}(\mathbb{R}^{n})}}} \nonumber\\
&\le \sum_{j=1}^{\infty}{\frac{2^{-\frac{mkj}{\alpha}+\delta k}2^{mkj}}{j!} \|\hat{f_{k}}}(\xi)\|_{L^{2}(\mathbb{R}^{2})} \nonumber\\
&\lesssim 2^{-s_{1}k}\|f\|_{H^{s_{1}+ \delta}(\mathbb{R}^{n})}.
\end{align}

Inequalities (\ref{Eq2.15}), (\ref{Eq2.16}) and (\ref{Eq2.17}) yield for $k \gg 1$,
\begin{align}\label{Eq2.18}
\biggl\|\mathop{sup}_{0<t<T} \frac{|e^{itP(D)}(f_{k})(\gamma(x,t))-f_{k}(x)|}{t^{\alpha \delta/m}}\biggl\|_{L^{q}(B(x_{0},R))} &\lesssim 2^{-\frac{\epsilon k}{2}}\|f\|_{H^{s_{1}+ \delta}(\mathbb{R}^{n})}.
\end{align}

It is clear that (\ref{Eq2.3}) follows from (\ref{Eq2.10}), (\ref{Eq2.11}) and (\ref{Eq2.18}).

\begin{theorem}\label{theorem2.1}
For each Schwartz function $f$, there exists
\[\gamma(x,t) = x - e_{1}t^{\alpha}, \:\ e_{1} = (1, 0,...,0), \:\ 0< \alpha \le 1,\]
such that if
\begin{equation}\label{Eq2.19}
\mathop{lim}_{t \rightarrow 0^{+}}{\frac{e^{itP(D)}(f)(\gamma(x,t))-f(x)}{t^{\alpha}}}=0, \:\ a.e. \:\ x \in \mathbb{R}^{n},
\end{equation}
then $f\equiv 0.$
\end{theorem}

\begin{proof}
 when $0 <\alpha <1$, by Taylor's formula,
\begin{align}
e^{itP(D)}f(\gamma(x,t)) - f(x) = t^{\alpha} \int_{\mathbb{R}^{n}}{e^{ix\cdot \xi}\xi_{1}\hat{f}(\xi)d\xi} +o(t^{2\alpha}) +o(t).
\end{align}
Therefore,
\begin{align}
\mathop{lim}_{t \rightarrow 0^{+}}{\frac{|e^{itP(D)}(f)(\gamma(x,t))-f(x)|}{t^{\alpha}}} \ge \frac{1}{2} \biggl| \int_{\mathbb{R}^{n}}{e^{ix\cdot \xi}\xi_{1}\hat{f}(\xi)d\xi} \biggl|.
\end{align}
If $f$ is not zero, then there is a set $A$ with positive measure and a constant $c>0$, such that
\[\biggl| \int_{\mathbb{R}^{n}}{e^{ix\cdot \xi}\xi_{1}\hat{f}(\xi)d\xi} \biggl| \ge c.\]
this contradicts (\ref{Eq2.19}). The same method is valid for $\alpha =1$.
\end{proof}

\section{Proof of Theorem \ref{theorem3.1}}\label{Section 2}
\textbf{Proof of Theorem \ref{theorem3.1}. }Set
\[A_{k} =: \{\xi: \frac{|\xi_{1}|}{2^{m_{2}k/m_{1}}} + \frac{|\xi_{2}|}{2^{k}} \sim 1\}.\]
Consider $k$ such that
\[\{\xi: |\xi| \sim \lambda\} \cap A_{k} \neq \emptyset,\]
then
\[\lambda^{m_{1}/m_{2}} \le  2^{k} \le \lambda,\]
therefore
\[k \sim log\lambda.\]
Decompose
\[f=\sum_{k}f_{k},\]
where supp $\hat{f}_{k} \subset \{\xi: |\xi| \sim \lambda\} \cap A_{k}.$ If for each $k$,
\begin{equation}\label{Eq3.5}
\biggl\|\mathop{sup}_{0<t<1} |e^{itP_{m_{1}, m_{2}}(D)}(f_{k})(\gamma(x,t))|\biggl\|_{L^{p}(B(x_{0},R))} \lesssim \lambda^{\alpha} \|f\|_{L^{2}},
\end{equation}
then
\begin{align}\label{Eq2.3.1}
&\biggl\|\mathop{sup}_{0<t<1} |e^{itP_{m_{1}, m_{2}}(D)}(f)(\gamma(x,t))|\biggl\|_{L^{p}(B(x_{0},R))} \nonumber\\
&\lesssim \sum_{k}\biggl\|\mathop{sup}_{0<t<1} |e^{itP_{m_{1}, m_{2}}(D)}(f_{k})(\gamma(x,t))|\biggl\|_{L^{p}(B(x_{0},R))} \nonumber\\
&\lesssim \lambda^{\alpha} \sum_{k}\|f\|_{L^{2}} \nonumber\\
&\lesssim \lambda^{\alpha + \epsilon} \|f\|_{L^{2}}.
\end{align}
It is sufficient to show that if
\begin{equation}\label{Eq2.3.1}
\biggl\|\mathop{sup}_{t \in J} |e^{itP_{m_{1}, m_{2}}(D)}(f_{k})(\gamma(x,t))|\biggl\|_{L^{p}(B(x_{0},R))} \lesssim \lambda^{\alpha} \|f\|_{L^{2}}
\end{equation}
for each $J \in \mathfrak{J}$, then (\ref{Eq3.5}) holds.

For each $g \in L^{2}$, $G(x,t) \in L_{x}^{p^{\prime}}(B(x_{0},R), L_{t}^{1}(0,1))$, $1/p+1/p^{\prime} =1$. Define the  operator $T$ by
\[Tg:= \int_{\mathbb{R}^{2}}{e^{i\gamma(x,t)\cdot \xi + itP_{m_{1},m_{2}}(\xi)} \Psi(\xi)\hat{g}(\xi) d\xi},\]
\[ G_{J}(x,t)=G(x,t)\chi_{J}(t),\]
where $\Psi(\xi_{1}, \xi_{2})=\psi(\frac{\xi_{1}}{2^{m_{2}k/m_{1}}}, \frac{\xi_{2}}{2^{k}})\psi(\frac{\xi_{1}}{\lambda}, \frac{\xi_{2}}{\lambda})$, $\psi \in C_{c}^{\infty}$ equals to $1$ on $\{\xi: |\xi| \sim 1\}$ and rapidly decay outside. $\chi_{J}(t)$ is the characteristic function of $J$. By duality, it is enough to show that
\begin{equation}\label{Eq3.8}
\|T^{\star}G_{J}\|_{L^{2}}  \lesssim \lambda^{\alpha} \|G_{J}\|_{L_{x}^{p^{\prime}}(B(x_{0},R), L_{t}^{1}(0,1))}, \:\ J \in \mathfrak{J}
\end{equation}
implies
\begin{equation}\label{Eq3.9}
\|T^{\star}G\|_{L^{2}}  \lesssim \lambda^{\alpha} \|G\|_{L_{x}^{p^{\prime}}(B(x_{0},R), L_{t}^{1}(0,1))}.
\end{equation}

Inequality (\ref{Eq3.9}) is equivalence to
\begin{equation}\label{Eq3.10}
\biggl|\sum_{J,J^{\prime}\in \mathfrak{J}} \langle T^{\star}G_{J}, T^{\star}G_{J^{\prime}} \rangle \biggl|  \lesssim \lambda^{2\alpha} \|G\|^{2}_{L_{x}^{p^{\prime}}(B(x_{0},R), L_{t}^{1}(0,1))}.
\end{equation}

When $dist(J,J^{\prime}) \ge 100 \lambda^{1-m_{1}}$,
\begin{align}\label{Eq3.11}
|\langle T^{\star}G_{J}, T^{\star}G_{J^{\prime}} \rangle |  &= |\langle G_{J}, \chi_{J} TT^{\star}G_{J^{\prime}} \rangle | \nonumber\\
&\le \|G_{J}\|_{L_{x}^{p^{\prime}}(B(x_{0},R), L_{t}^{1}(0,1))} \|\chi_{J}TT^{\star}G_{J^{\prime}}\|_{L_{x}^{p}(B(x_{0},R), L_{t}^{\infty}(0,1))} \nonumber\\
&\le \|G_{J}\|_{L_{x}^{p^{\prime}}(B(x_{0},R), L_{t}^{1}(0,1))} \|\chi_{J}TT^{\star}G_{J^{\prime}}\|_{L_{x}^{\infty}(B(x_{0},R), L_{t}^{\infty}(0,1))}.
\end{align}
For each $x \in B(x_{0},R)$, $t \in J$,
\begin{equation}\label{Eq3.12}
|\chi_{J}TT^{\star}G_{J^{\prime}} (x,t)| \le \int_{B(x_{0,R})} {\int_{0}^{1} |K(x,y,t,t^{\prime})||G_{J^{\prime}}(y,t^{\prime})|dt^{\prime}dy, }
\end{equation}
where
\begin{align}
 |K(x,y,t,t^{\prime})| &= \biggl| \int_{\mathbb{R}^{2}}{e^{i[\gamma(x,t)-\gamma(y,t^{\prime})]\cdot (\xi_{1}, \xi_{2}) + i(t-t^{\prime})P_{m_{1},m_{2}}(\xi_{1}, \xi_{2})} \Psi^{2}(\xi_{1},\xi_{2}) d\xi_{1}d\xi_{2}} \biggl| \nonumber\\
&= 2^{(\frac{m_{2}}{m_{1}}+1)k} \biggl| \int_{\mathbb{R}^{2}}{e^{i\Phi_{x,y,t,t^{\prime}}(\eta_{1},\eta_{2})} \Psi^{2}(2^{\frac{m_{2}}{m_{1}}k}\eta_{1}, 2^{k}\eta_{2}) d\eta_{1}d\eta_{2}} \biggl|,
\end{align}
in which
\[\Phi_{x,y,t,t^{\prime}}(\eta_{1},\eta_{2}) = [\gamma(x,t)-\gamma(y,t^{\prime})]\cdot (2^{\frac{m_{2}}{m_{1}}k}\eta_{1}, 2^{k}\eta_{2}) + 2^{m_{2}k}(t-t^{\prime})P_{m_{1},m_{2}}(\eta_{1}, \eta_{2}).\]
Since
\[|t-t^{\prime}| \ge 100 \lambda^{1-m_{1}} \ge 100 2^{(\frac{m_{2}}{m_{1}}-m_{2})k}, \:\ |\eta_{1}|+|\eta_{2}| \sim 1,\]
it is easy to check that
\[|\nabla_{\eta} \Phi_{x,y,t,t^{\prime}}(\eta_{1},\eta_{2})| \gtrsim 2^{m_{2}k}|t-t^{\prime}|, \]
\[|D_{\eta}^{\beta} \Phi_{x,y,t,t^{\prime}}(\eta_{1},\eta_{2})| \lesssim  2^{m_{2}k}|t-t^{\prime}|, |\beta| \ge 2.\]
Integration by parts implies that for positive integer $N \gg \frac{m_{2}}{m_{1}}$,
\begin{equation}\label{Eq3.14}
 |K(x,y,t,t^{\prime})| \le 2^{(\frac{m_{2}}{m_{1}}+1)k} \frac{C_{N}}{(1+2^{m_{2}k}|t-t^{\prime}|)^{N}} (\frac{2^{m_{2}k/m_{1}}}{\lambda})^{N} \lesssim \lambda^{-O(N)}.
\end{equation}

Inequalities (\ref{Eq3.11}), (\ref{Eq3.12}), (\ref{Eq3.14}) and H\"{o}lder's inequality imply
\begin{equation}\label{Eq3.15}
|\langle T^{\star}G_{J}, T^{\star}G_{J^{\prime}} \rangle |   \lesssim \lambda^{-O(N)}\|G_{J}\|_{L_{x}^{p^{\prime}}(B(x_{0},R), L_{t}^{1}(0,1))} \|G_{J^{\prime}}\|_{L_{x}^{p^{\prime}}(B(x_{0},R), L_{t}^{1}(0,1))}
\end{equation}
when $dist(J,J^{\prime}) \ge 100 \lambda^{1-m_{1}}$.

For $J,J^{\prime}$ such that $dist(J,J^{\prime}) \le 100 \lambda^{1-m_{1}}$, by H\"{o}lder's inequality and (\ref{Eq3.8}) we have
\begin{align}\label{Eq3.16}
|\langle T^{\star}G_{J}, T^{\star}G_{J^{\prime}} \rangle |   &\lesssim \|T^{\star}G_{J}\|_{L^{2}} \|T^{\star}G_{J^{\prime}}\|_{L^{2}} \nonumber\\
&\lesssim \lambda^{2\alpha}\|G_{J}\|_{L_{x}^{p^{\prime}}(B(x_{0},R), L_{t}^{1}(0,1))} \|G_{J^{\prime}}\|_{L_{x}^{p^{\prime}}(B(x_{0},R), L_{t}^{1}(0,1))}.
\end{align}

It follows from (\ref{Eq3.15}), (\ref{Eq3.16}) and $1 \le p^{\prime} \le 2$ that
\begin{align}
\biggl|\sum_{J,J^{\prime}\in \mathfrak{J}} \langle T^{\star}G_{J}, T^{\star}G_{J^{\prime}} \rangle \biggl|  &\lesssim \sum_{J,J^{\prime}\in \mathfrak{J}, dist(J,J^{\prime}) \le 100 \lambda^{1-m_{1}}}|\langle T^{\star}G_{J}, T^{\star}G_{J^{\prime}} \rangle | \nonumber\\
&\:\:\ + \sum_{J,J^{\prime}\in \mathfrak{J}, dist(J,J^{\prime}) \ge 100 \lambda^{1-m_{1}}}|\langle T^{\star}G_{J}, T^{\star}G_{J^{\prime}} \rangle | \nonumber\\
&\lesssim \lambda^{2\alpha} \sum_{J \in \mathfrak{J}}\|G_{J}\|^{2}_{L_{x}^{p^{\prime}}(B(x_{0},R), L_{t}^{1}(0,1))} \nonumber\\
&\lesssim \lambda^{2\alpha} \|G\|^{2}_{L_{x}^{p^{\prime}}(B(x_{0},R), L_{t}^{1}(0,1))}, \nonumber
\end{align}
which implies (\ref{Eq3.10}).

\end{document}